\documentclass[a4paper, reqno]{amsart}
\usepackage{amssymb, amsmath, amsthm, chngcntr, enumerate, mathabx, mathrsfs, mathtools}
\usepackage[numbers]{natbib}

\newtheorem{thm}{Theorem}

\newtheorem{corollary}[thm]{Corollary}
\newtheorem{lemma}[thm]{Lemma}
\newtheorem{proposition}[thm]{Proposition}

\theoremstyle{remark}
\newtheorem{rmk}[thm]{Remark}

\newcommand{\C}{\mathbb{C}}
\newcommand{\R}{\mathbb{R}}
\newcommand{\Z}{\mathbb{Z}}
\newcommand{\N}{\mathbb{N}}
\newcommand{\T}{\mathbb{T}}
\newcommand{\D}{\mathbb{D}}

\title[Extensions of a theorem of Pichorides]{Multi-parameter extensions of a theorem of Pichorides}
\author[Bakas, Rodr\'iguez-L\'opez, and Sola]{Odysseas Bakas, Salvador Rodr\'iguez-L\'opez, and Alan Sola}
\address{Department of Mathematics, Stockholm University, 106 91 Stockholm, Sweden}
\email{bakas@math.su.se, s.rodriguez-lopez@math.su.se, sola@math.su.se}

\thanks{The second author is partially supported by the Spanish Government grant MTM2016-75196-P}

\begin{document}

\begin{abstract}
Extending work of Pichorides and Zygmund to the $d$-dimensional setting, we show that the supremum of $L^p$-norms of the Littlewood-Paley square function over the unit ball of the analytic Hardy spaces $H^p_A(\T^d)$ blows up like $(p-1)^{-d}$ as $p\to 1^+$. Furthermore, we obtain an $L\log^d L$-estimate for square functions on $H^1_A(\T^d)$. Euclidean variants of Pichorides's theorem are also obtained.
\end{abstract}

\maketitle

\section{Introduction}
Given a trigonometric polynomial $f$ on $\T$, we define the classical Littlewood-Paley square function $S_{\T} (f)$ of $f$ by
$$  S_{\T} (f) (x) = \big( \sum_{k\in \Z} | \Delta_k (f) (x) |^2 \big)^{1/2},$$
where for $k \in \N$, we set
$$\Delta_k (f) (x) = \sum_{n=2^{k-1}}^{2^k-1} \widehat{f}(n) e^{i 2 \pi n x} \quad \mathrm{and} \quad \Delta_{-k} (f) (x) = \sum_{n=-2^k+1 }^{-2^{k-1} } \widehat{f}(n) e^{i 2 \pi n x}$$
 and for $k = 0$ we take $\Delta_0 (f) (x) = \widehat{f} (0)$.

A classical theorem of J.E. Littlewood and R.E.A.C. Paley asserts that for every $1<p<\infty$ there exists a constant $B_p >0$ such that
\begin{equation}\label{L-P_classical_ineq}
 \| S_{\T} (f) \|_{L^p (\T)} \leq B_p \| f \|_{L^p (\T)}
\end{equation}
for every trigonometric polynomial $f$ on $\T$, see, e.g., \cite{Edwards_Gaudry} or \cite{Zygmund_book}.

The operator $S_{\T}$ is not bounded on $L^1 (\T)$, and hence, the constant $B_p$ in (\ref{L-P_classical_ineq}) blows up as $p \rightarrow 1^+$. In \cite{Bourgain}, J. Bourgain obtained the sharp estimate
\begin{equation}\label{as_1} 
B_p \sim (p-1)^{-3/2} \quad \textrm{as}\quad p\to 1^+.
\end{equation}
For certain subspaces of $L^p(\T)$, however, one might hope for better bounds. In \cite{Pichorides}, S. Pichorides showed that for the analytic Hardy spaces $H_A^p(\T)$ ($1<p \leq 2$), we have
\begin{equation}\label{Pichorides_thm}
\sup_{\substack{  \| f \|_{L^p (\T)} \leq 1 \\  f \in H_A^p (\T) } }  \| S_{\T} (f) \|_{L^p (\T) } \sim (p-1)^{-1} \quad \textrm{as} \quad p\to 1^+.
\end{equation}

Higher-dimensional extensions of Bourgain's result (\ref{as_1}) were obtained by the first author in \cite{Bakas}. In particular, given a dimension $d \in \N$, if $f$ is a trigonometric polynomial on $\T^d$, we define its $d$-parameter Littlewood-Paley square function by
$$ S_{\T^d} (f) (x) = \big( \sum_{k_1, \cdots, k_d \in \Z} | \Delta_{k_1, \cdots, k_d} (f) (x) |^2 \big)^{1/2} ,$$
where for $k_1, \cdots, k_d \in \Z$ we use the notation 
$ \Delta_{k_1, \cdots, k_d} (f)  = \Delta_{k_1} \otimes \cdots \otimes \Delta_{k_d} (f)$, where the direct product notation indicates the operator $\Delta_{k_j}$ in the $j$-th position acting on the $j$-th variable. As in the one-dimensional case, for every $1<p<\infty$, there is a positive constant $B_p (d)$ such that
$$ \| S_{\T^d} (f) \|_{L^p (\T^d)} \leq B_p (d) \| f \|_{L^p (\T^d)}$$
for each trigonometric polynomial $f$ on $\T^d$. It is shown in \cite{Bakas} that
\begin{equation}\label{as_1_n}
B_p (d) \sim_d (p-1)^{-3d/2} \quad \textrm{as} \quad p\rightarrow 1^+.
\end{equation}

A natural question in this context is whether one has an improvement on the limiting behaviour of $B_p (d)$ as $p \rightarrow 1^+$ when restricting to the analytic Hardy spaces $H^p_A (\T^d)$. In other words, one is led to ask whether the aforementioned theorem of Pichorides can be extended to the polydisc. However, the proof given in \cite{Pichorides} relies on factorisation of Hardy spaces, and it is known, see for instance \cite[Chapter 5]{Rudin} and \cite{RubShi}, that canonical factorisation fails in higher dimensions. 

In this note we obtain an extension of (\ref{Pichorides_thm}) to the polydisc as a consequence of a more general result involving tensor products of Marcinkiewicz multiplier operators on $\T^d$. Recall that, in the periodic setting, a multiplier operator $T_m$ associated to a function $m \in \ell^{\infty} (\Z)$ is said to be a Marcinkiewicz multiplier operator on the torus if $m$ satisfies
\begin{equation}\label{Marcinkiewicz_torus}
 B_m  = \sup_{k \in \N} \big[  \sum_{n=2^k -1}^{2^{k+1}} |m(n+1) - m(n)| +   \sum_{n=-2^{k+1}}^{-2^k +1} |m(n+1) - m(n)|  \big]  < \infty .
\end{equation}
Our main result in this paper is the following theorem.

\begin{thm}\label{Main_Theorem}
Let $d \in \N$ be a given dimension. If $T_{m_j}$ is a Marcinkiewicz multiplier operator
on $\T$ $(j=1, \cdots, d )$, then for every $ f \in H_A^p (\T^d)$ one has
$$   \| ( T_{m_1} \otimes \cdots \otimes T_{m_d} )  (f) \|_{L^p (\T^d) } \lesssim_{C_{m_1}, \cdots, C_{m_d} } (p-1)^{-d} \| f \|_{L^p (\T^d)} $$
as $p \rightarrow 1^+$, where $C_{m_j} = \| m_j \|_{\ell^{\infty} (\Z)} + B_{m_j}$, 
$B_{m_j}$ being as in $(\ref{Marcinkiewicz_torus})$, $j=1, \cdots, d$.
\end{thm}
 
To prove Theorem \ref{Main_Theorem}, we use a theorem of T. Tao and J. Wright \cite{TW} on the endpoint mapping properties of Marcinkiewicz multiplier operators on the line, transferred to the periodic setting, with a variant of Marcinkiewicz interpolation for Hardy spaces which is due to S. Kislyakov and Q. Xu \cite{KX}. 

Since for every choice of signs the randomised version $ \sum_{k \in \Z} \pm  \Delta_k$ of $S_{\T}$ is a Marcinkiewicz multiplier operator on the torus with corresponding constant $B_m \leq 2$, Theorem \ref{Main_Theorem} and Khintchine's inequality yield the following $d$-parameter extension of Pichorides's theorem \eqref{Pichorides_thm}.
 
\begin{corollary}\label{Pichorides_extension}
Given $d \in \N$, if $S_{\T^d}$ denotes the $d$-parameter Littlewood-Paley square function, then one has
$$ \sup_{\substack{  \| f \|_{L^p (\T^d)} \leq 1 \\  f \in H_A^p (\T^d) } }  \| S_{\T^d} (f) \|_{L^p (\T^d) } \sim_d (p-1)^{-d} \quad \textrm{as}\quad p\to 1^ +.$$
\end{corollary}

The present paper is organised as follows: In the next section we set down notation and provide some background, and in Section \ref{Proof_of_Thm} we prove our main results. Using the methods of Section \ref{Proof_of_Thm}, in Section \ref{Application} we extend a well-known inequality due to A. Zygmund \cite[Theorem 8]{Zygmund} to higher dimensions. In the last section we obtain a Euclidean version of Theorem \ref{Main_Theorem} by using the aforementioned theorem of Tao and Wright \cite{TW} combined with a theorem of Peter Jones \cite{PeterJones} on a Marcinkiewicz-type decomposition for analytic Hardy spaces over the real line.

\section{Preliminaries}

\subsection{Notation}
We denote the set of natural numbers by $\N$, by $\N_0$ the set of non-negative integers and by $\Z$ the set of integers.

Let $f$ be a function of $d$-variables. Fixing the first $d-1$ variables $(x_1, \cdots, x_{d-1})$, we write $f(x_1, \cdots, x_d) = f_{(x_1, \cdots, x_{d-1})} (x_d)$.
The sequence of the Fourier coefficients of a function  $f\in L^1(\T^d)$ will be denoted by $\hat{f}$.

Given a function $m \in L^{\infty} (\R^d)$, we denote by $T_m$ the multiplier operator corresponding to $m$, initially defined on $L^2 (\R^d)$, by $(T_m (f) )^{\widehat{\ }} (\xi) = m(\xi) \widehat{f} (\xi)$, $\xi \in \R^d$. Given $\mu \in \ell^{\infty} (\Z^d) $, one defines (initially on $L^2 (\T^d)$) the corresponding periodic multiplier operator $T_{\mu}$ in an analogous way.

 If $\lambda $ is a continuous and bounded function on the real line and $T_{\lambda} $ is as above, $T_{\lambda |_{\Z}}$ denotes the periodic multiplier operator such that $T_{\lambda |_{\Z}} (f) (x) = \sum_{n \in \Z} \lambda (n) \widehat{f} (n) e^{i 2\pi n x}$ ($x \in \T$) for every trigonometric polynomial $f$ on $\T$.

Given two positive quantities $X$ and $Y$ and a parameter $\alpha$, we write $X \lesssim_{\alpha} Y$ (or simply $X \lesssim Y$) whenever there exists a constant $C_{\alpha} >0$ depending on $\alpha$ so that $X \leq C_{\alpha} Y$. If $X \lesssim_{\alpha} Y$ and $Y \lesssim_{\alpha} X$, we write $X \sim_{\alpha} Y$ (or simply $X \sim Y$).  

\subsection{Hardy spaces and Orlicz spaces}\label{Hardy_Orlicz}

Let $d \in \N$. For $0<p<\infty$, let $H^p_A (\D^d)$ denote the space of holomorphic functions $F$ on $\D^d$, $\D = \{ z \in \C : |z| < 1 \}$, such that
$$ \| F \|_{H^p_A (\D^d)}^p = \sup_{0 \leq r_1 , \cdots, r_d <1} \int_{\T^d} |F ( r_1 e^{i 2 \pi x_1}, \cdots, r_d e^{i 2 \pi x_d} )|^p d x_1 \cdots d x_d < \infty. $$ 
For $p= \infty$, $H^{\infty}_A (\D^d)$ denotes the class of bounded holomorphic functions on $\D^d$. It is well-known that for $1 \leq p \leq \infty$, the limit $f$ of $F \in H^p_A (\D^d) $ as we approach the distinguished boundary $\T^d$ of $\D^d$, namely
$$ f(x_1, \cdots, x_d) = \lim_{r_1, \cdots, r_d \rightarrow 1^-} F ( r_1 e^{i 2 \pi x_1}, \cdots, r_d e^{i 2 \pi x_d} )$$ 
exists a.e. in $\T^d$ and $\| F \|_{H^p_A (\D^d)} = \| f \|_{L^p (\T^d)}$. For $1 \leq p \leq \infty$, we define the analytic Hardy space $H^p_A (\T^d)$ on the $d$-torus as the space of all functions in $L^p (\T^d)$ that are boundary values of functions in $H^p_A (\D^d)$. Moreover, it is a standard fact that $ H_A^p (\T^d) = \{ f \in L^p (\T^d)\colon \ \mathrm{supp} (\widehat{f}) \subset \N_0^d \} $. 
Hardy spaces are discussed in Chapter 7 in \cite{Duren}, where the case $d=1$ is treated, and in Chapter 3 of \cite{Rudin}.

If $f \in L^1 (\T^d)$ is such that $\mathrm{supp} (  \widehat{f}) $ is finite, then $f$ is said to be a trigonometric polynomial on $\T^n$, and if moreover $\mathrm{supp} ( \widehat{f} ) \subset \N^d_0$, then $f$ is said to be analytic. It is well-known \cite{Duren, Rudin} that for $1\leq p < \infty$, the class of trigonometric polynomials on $\T^d$ is a dense subspace of $L^p (\T^d)$ and analytic trigonometric polynomials on $\T^d$ are dense in $H^p_A (\T^d)$. 

We define the real Hardy space $H^1 (\T)$ to be the space of all integrable functions $f \in L^1 (\T)$ such that $H_{\T} (f) \in L^1 (\T)$, where $H_{\T} (f) $ denotes the periodic Hilbert transform of $f$. One  sets $\| f \|_{H^1 (\T)}  = \| f \|_{L^1 (\T)} + \| H_{\T} (f) \|_{L^1 (\T)}$. 
Note that $H^1_A (\T)$ can be regarded as a proper subspace of $H^1 (\T)$ and moreover, $\| f \|_{H^1 (\T)} = 2 \| f \|_{L^1 (\T)}$ when $f \in H^1_A (\T)$.

Given $d\in \N$, for $0<p<\infty$, let $H^p_A ((\R^2_+)^d)$ denote the space of holomorphic functions $F$ on $(\R^2_+)^d $, where $\R^2_+ = \{ x + iy \in \C : y >0 \}$, such that
$$ \| F \|_{H^p_A ((\R^2_+)^d)}^p = \sup_{y_1, \cdots, y_d >0} \int_{\R^d} |F(x_1 + i y_1, \cdots, x_d + i y_d) |^p  dx_1 \cdots d x_d < \infty. $$
For $p = \infty$, $H^{\infty}_A ((\R^2_+)^d)$ is defined as the space of bounded holomorphic functions in $(\R^2_+)^d$.
For $1 \leq p \leq \infty$, for every $F \in H^p_A ((\R^2_+)^d)$ its limit $f$ as we approach the boundary $\R^d$, namely 
$$f (x_1, \cdots, x_d) = \lim_{y_1, \cdots, y_d \rightarrow 0^+} F (x_1 + i y_1, \cdots, x_d + i y_d),$$
 exists for a.e. $(x_1, \cdots, x_d)  \in \R^d$ and, moreover, $\| F \|_{H^p_A ((\R^2_+)^d)} = \| f \|_{L^p (\R^d)}$. Hence, as in the periodic setting, for $1 \leq p \leq \infty$ we may define the $d$-parameter analytic Hardy space $H^p_A (\R^d)$ to be the space of all functions in $L^p (\R^d)$ that are boundary values of functions in $H^p_A ((\R^2_+)^d)$.

The real Hardy space $H^1 (\R)$ on the real line is defined as the space of all integrable functions $f$ on $\R$ such that $H (f) \in L^1 (\R)$, where $H(f)$ is the Hilbert transform of $f$. Moreover, we set $\| f \|_{H^1 (\R)} = \| f \|_{L^1 (\R)} + \| H (f) \|_{L^1 (\R)}$.

We shall also consider the standard Orlicz spaces $L \log^r L (\T^d)$. 
For $r >0$, one may define $L \log^r L (\T^d)$ as the space of measurable functions $f$ on $\T^d$ such that $ \int_{\T^d} |f(x)| \log^r (1+|f(x)|) dx < \infty$. For $r \geq 1$, we may equip $L \log^r L (\T^d)$ with a norm given by $$ \| f \|_{L \log^r L (\T^d)} = \inf \big\{ \lambda >0 : \int_{\T^d} \frac{|f(x)|}{\lambda} \log^r \big( 1+\frac{|f(x)|}{\lambda} \big) dx \leq 1  \big\}. $$
For more details on Orlicz spaces, we refer the reader to the books \cite{KR} and \cite{Zygmund_book}.

\section{Proof of Theorem \ref{Main_Theorem}}\label{Proof_of_Thm}

Recall that a function $m \in L^{\infty} (\R)$ is said to be a Marcinkiewicz multiplier on $\R$ if it is differentiable in every dyadic interval $\pm [2^k, 2^{k+1})$, $k \in \Z$ and 
\begin{equation}\label{Marcinkiewicz_condition}
 A_m  = \sup_{k \in \Z} \big[  \int_{[2^k ,2^{k+1})} |m' (\xi) | d \xi +  \int_{(-2^{k+1} , - 2^k ]} |m' (\xi) | d \xi \big]  < \infty 
\end{equation} 

If $ m \in L^{\infty} (\R)$ satisfies (\ref{Marcinkiewicz_condition}), then thanks to a classical result of J. Marcinkiewicz\footnote{Marcinkiewicz originally proved the theorem in the periodic setting, see \cite{Marcinkiewicz}.}, see e.g. \cite{Stein}, the corresponding multiplier operator $T_m$ is bounded on $L^p (\R)$ for all $1< p < \infty$.  In \cite{TW}, Tao and Wright showed that every Marcinkiewicz multiplier operator $T_m$ is bounded from the real Hardy space $H^1 (\R)$ to $L^{1, \infty} (\R)$, namely
\begin{equation}\label{TW_ineq} 
\| T_m (f) \|_{L^{1,\infty} (\R)} \leq C_m \| f \|_{H^1 (\R)},
\end{equation}
where the constant $C_m$ depends only on $ \| m \|_{L^{\infty} (\R)} + A_m$, with $A_m$ as in (\ref{Marcinkiewicz_condition}). For the sake of completeness, let us recall that $L^{1,\infty}(\mathcal{M})$ stands for the quasi-Banach space of measurable functions on a measure space $\mathcal{M}$ endowed with the quasinorm  
$$	\|f\|_{L^{1,\infty}(\mathcal{M})}:= \sup_{t>0} t \, \int_{\{x\in \mathcal{M}: |f(x)|>t\}} dx. $$

Either by adapting the proof of Tao and Wright to the periodic setting or by using {a transference argument}, see Subsection \ref{transference}, one deduces that every periodic Marcinkiewicz multiplier operator $T_m$ satisfies
\begin{equation}\label{TW_periodic} 
\| T_m (f) \|_{L^{1,\infty} (\T)} \leq D_m \| f \|_{H^1 (\T)},
\end{equation}
where $D_m$ depends on $\| m \|_{\ell^{\infty}(\Z)} + B_m$, $B_m$ being as in (\ref{Marcinkiewicz_torus}). Therefore, it follows that for every $f \in H^1_A (\T)$ one has
\begin{equation}\label{TW_analytic} 
\| T_m (f) \|_{L^{1,\infty} (\T)} \leq D'_m \| f \|_{L^1 (\T)},
\end{equation}
where one may take $D'_m = 2 D_m$. 
We shall prove that for every Marcinkiewicz multiplier operator $T_m$ on the torus one has
\begin{equation}\label{Marcinkiewicz_norm}
  \sup_{\substack{  \| f \|_{L^p (\T)} \leq 1 \\  f \in H_A^p (\T) } }  \| T_m (f) \|_{L^p (\T) } \lesssim_{B_m} (p-1)^{-1}  
 \end{equation}
as $p \rightarrow 1^+$. To do this, we shall make use of the following lemma due to Kislyakov and Xu \cite{KX}.

\begin{lemma}[Kislyakov and Xu, \cite{KX} and \cite{Pavlovic}] \label{KX_P}
If $f \in H^{p_0}_A (\T)$ $(0 < p_0 < \infty)$ and $\lambda >0$, then there exist functions $h_{\lambda} \in H^{\infty}_A (\T)$, $g_{\lambda} \in H^{p_0}_A (\T) $ and a constant $C_{p_0}>0$ depending only on $p_0$ such that
\begin{itemize}
\item $ |h_{\lambda}(x)| \leq C_{p_0} \lambda \min \{ \lambda^{-1} |f(x)|, |f(x)|^{-1} \lambda \} $ for all $x \in \T$,
\item $ \| g_{\lambda} \|^{p_0}_{L^{p_0} (\T)} \leq C_{p_0} \int_{\{ x \in \T : |f (x)| > \lambda \} } |f (x)|^{p_0} dx$, and
\item $f =h_{\lambda} + g_{\lambda}$.
\end{itemize}
\end{lemma}
We remark that by examining the proof of Lemma \ref{KX_P}, one deduces that when $1\leq p_0 \leq 2$ the constant $C_{p_0} $ in the statement of the lemma can be chosen independent of $p_0$. To prove the desired inequality (\ref{Marcinkiewicz_norm}) and hence Theorem \ref{Main_Theorem} in the one-dimensional case, we argue as in \cite[Theorem 7.4.1]{Pavlovic}. More precisely, given a $1<p<2$, if $f$ is a fixed analytic trigonometric polynomial on $\T$, we first write $ \| T_m (f) \|_{L^p (\T)}^p = p \int_0^{\infty} \lambda^{p-1} |\{ x \in \T : |T_m (f) (x) | > \lambda \}| d \lambda $
and then by using Lemma \ref{KX_P} for $p_0 =1 $ we obtain  $ \| T_m (f) \|_{L^p (\T)}^p \leq I_1 + I_2 $, where 
$$ I_1 = p \int_0^{\infty} \lambda^{p-1} |\{ x \in \T : |T_m (g_{\lambda}) (x) | > \lambda/2 \}| d \lambda  $$
and 
$$ I_2 = p \int_0^{\infty} \lambda^{p-1} |\{ x \in \T : |T_m (h_{\lambda}) (x) | > \lambda/2 \}| d \lambda .$$
To handle $I_1$, we use the boundedness of $T_m$ from $H^1_A (\T)$ to $L^{1,\infty} (\T)$ and Fubini's theorem to deduce that
$$ I_1 \lesssim (p-1)^{-1} \int_{\T} |f (x)|^p dx.$$
To obtain appropriate bounds for $I_2$, we use the boundedness of $T_m$ from $H^2_A (\T)$ to $L^2 (\T)$ and get
$$ I_2 \lesssim  \int_0^{\infty} \lambda^{p-3} (\int_{\{ x \in \T : |f(x)| \leq \lambda \}} |f(x)|^2 dx ) d \lambda + \int_0^{\infty} \lambda^{p+1} (\int_{\{ x \in \T : |f(x)| > \lambda \} } |f(x)|^{-2} dx ) d \lambda. $$ 
Hence, by applying Fubini's theorem to each term, we obtain
$$ I_2 \lesssim (2-p)^{-1} \int_{\T} |f(x)|^p dx + (p+2)^{-1} \int_{\T} |f(x)|^p dx .$$
Combining the estimates for $I_1$ and $I_2$ and using the density of analytic trigonometric polynomials in $( H^p_A (\T), \| \cdot \|_{L^p (\T)})$, (\ref{Marcinkiewicz_norm}) follows.

To prove the $d$-dimensional case, take $f$ to be an analytic trigonometric polynomial on $\T^d$ and note that if $T_{m_j}$ are periodic Marcinkiewicz multiplier operators ($j=1,\cdots,d$), then for fixed $(x_1, \cdots, x_{d-1}) \in \T^{d-1}$ one can write
$$  T_{m_d} (g_{(x_1, \cdots, x_{d-1})}) (x_d) = T_{m_1} \otimes \cdots \otimes T_{m_d} (f) (x_1, \cdots, x_d) ,   $$
where
$$g_{(x_1, \cdots, x_{d-1})} (x_d) = T_{m_1} \otimes \cdots \otimes T_{m_{d-1}} (f_{(x_1, \cdots, x_{d-1})}) (x_d). $$
Hence, by using (\ref{Marcinkiewicz_norm}) in the $d$-th variable, one deduces that
$$ \| T_{m_n} (g_{(x_1, \cdots, x_{d-1})})  \|^p_{L^p (\T)} \leq C_{m_d}^p (p-1)^{-p} \| g_{(x_1, \cdots, x_d)}\|^p_{L^p (\T)}$$
where $C_{m_d} >0$ is the implied constant in (\ref{Marcinkiewicz_norm}) corresponding to $T_{m_d}$. By iterating this argument $d-1$ times, one obtains
$$ \| T_{m_1} \otimes \cdots \otimes T_{m_d} (f)   \|^p_{L^p (\T^d)} \leq [C_{m_1} \cdots C_{m_d}]^p (p-1)^{-dp} \| f \|^p_{L^p (\T^d)}$$
and this completes the proof of Theorem \ref{Main_Theorem}.

\subsection{Proof of Corollary \ref{Pichorides_extension}}\label{Corollary_of_Thm}
We shall use the multi-dimensional version of Khintchine's inquality: if $(r_k)_{k \in \N_0}$ denotes the set of Rademacher functions indexed by $\N_0$ over a probability space $(\Omega, \mathcal{A}, \mathbb{P})$, then for every finite collection of complex numbers $( a_{k_1, \cdots, k_d})_{k_1, \cdots, k_d \in \N_0}$ one has
\begin{equation}\label{multi_Khintchine}
\Big\| \sum_{k_1, \cdots, k_d \in \N_0}  a_{k_1, \cdots, k_d}   r_{k_1} \otimes \cdots \otimes r_{k_d} \Big\|_{L^p (\Omega^d)} \sim_p \big( \sum_{k_1, \cdots, k_d \in \N_0} |a_{k_1, \cdots, k_d}|^2 \big)^{1/2} 
\end{equation}
for all $0<p<\infty $. The implied constants do not depend on $( a_{k_1, \cdots, k_d})_{k_1, \cdots, k_d \in \N_0}$, see e.g. Appendix D in \cite{Stein}, and do not blow up as $p \rightarrow 1$.

Combining Theorem \ref{Main_Theorem}, applied to $d$-fold tensor products of periodic Marcinkiewicz multiplier operators of the form $ \sum_{k \in \Z} \pm \Delta_k$, with the multi-dimensional Khintchine's inequality as in \cite[Section 3]{Bakas} shows that the desired bound holds for analytic polynomials. Since analytic trigonometric polynomials on $\T^d$ are dense in $( H^p_A (\T^d), \| \cdot \|_{L^p (\T^d)})$, we deduce that
$$ \sup_{ \substack{ \| f \|_{L^p (\T^d)} \leq 1 \\ f \in H^p_A (\T^d) } } \| S_{\T^d} (f) \|_{L^p (\T^d)} \lesssim_d (p-1)^{-d} \quad \textrm{as}\quad p \to 1^ +.$$

It remains to prove the reverse inequality.
To do this, for fixed $1<p \leq 2$, choose an $f \in H^p_A (\T)$ such that 
$$ \| S_{\T} (f) \|_{L^p (\T)} \geq C (p-1)^{-1} \| f \|_{L^p (\T)},$$
where $C>0$ is an absolute constant. The existence of such functions is shown in \cite{Pichorides}. Hence, if we define $g \in H^p_A (\T^d)$ by 
$$g  (x_1, \cdots, x_d) = f(x_1) \cdots f (x_d)$$ 
for $(x_1, \cdots, x_d )\in \T^d$, then
\begin{align*}
 \| S_{\T^d} (g) \|_{L^p (\T^d)} &= \| S_{\T} (f) \|_{L^p (\T)} \cdots \|S_{\T} (f) \|_{L^p (\T)}  \\
 &\geq C^d (p-1)^{-d} \| f \|^d_{L^p (\T)}  =C^d (p-1)^{-d} \| g \|_{L^p (\T^d)} 
 \end{align*}
and this proves the sharpness of Corollary \ref{Pichorides_extension}.
\begin{rmk}For the subspace $H^p_{A,\mathrm{diag}}(\T^d)$ of $L^p (\T^d)$ consisting of functions of the form $f(x_1,\ldots, x_d)=F(x_1+\cdots + x_d)$ for some one-variable function $F\in H^p_A(\T)$, we have 
the improved estimate
\[\sup_{\substack{\|f\|_{L^p(\T^d)}\leq1\\ f \in H_{A, \mathrm{diag}}^p(\T^d)}}  \|S_{\mathbb{T}^d}(f)\|_p\sim (p-1)^{-1}, \quad p\to 1^+.\] 
This follows from invariance of the $L^p$-norm and Fubini's theorem which allow us to reduce to the one-dimensional case. On the other hand, the natural inclusion of $H^p_A(\T^k)$ in $H^p_A(\T^d)$  
yields examples of subspaces with sharp blowup of order $(p-1)^{-k}$ for any $k=1,\ldots, d-1$.

Both the original proof of Pichorides's theorem and the extension in this paper rely on complex-analytic techniques, via canonical factorisation in \cite{Pichorides} and conjugate functions in \cite{Pavlovic}. However, a complex-analytic structure is not necessary in order for an estimate of the form in Corollary \ref{Pichorides_extension} to hold. For instance, the same conclusion remains valid for $g\in \tilde{H}^p_0(\T^d)$, the subset of $L^p(\T^d)$ consisting of functions with $\mathrm{supp}(\hat{f})\subset (-\mathbb{N})^d$. Moreover, for functions of the form $f+g$, 
where $f \in H^p_A(\T^d)$ and $g \in \tilde{H}^p_0(\T^d)$ with $\|f\|_{p}=\|g\|_{p}\leq 1/2$, we then have $\|f+g\|_{p}\leq 1$ and 
$\|S_{\mathbb{T}^n}(f+g)\|_{p}\leq  \|S_{\mathbb{T}^d}(f)\|_{p}+\|S_{\mathbb{T}^d}(g)\|_p\lesssim (p-1)^{-d}$ as $p\to 1^+$.
\end{rmk}

\subsection{A transference theorem}\label{transference}
In this subsection, we explain how one can transfer the aforementioned result of Tao and Wright on boundedness of Marcinkiewicz multiplier operators from $H^1 (\R)$ to $L^{1,\infty} (\R)$ to the periodic setting. To this end, let us first recall the definition of the local Hardy space $h^1 (\R)$ introduced by D. Goldberg \cite{Goldberg}, which can be described as the space of $L^1$-functions for which the ``high-frequency''part  belongs to $H^1 (\R)$. Namely, if we take $\phi $ to be a smooth function supported in $ [-1,1]$ and such that $\phi |_{[-1/2, 1/2]} \equiv 1$, and set $\psi = 1-\phi$, one has that $f\in h^1 (\R)$ if, and only if, 
$$	 \| f \|_{h^1 (\R)}=\| f \|_{L^1 (\R)}+\|T_{\psi}  ( f ) \|_{H^1 (\R)}<+\infty. $$
The desired transference result is a consequence of D. Chen's \cite[Thm. 29]{Chen}. 

\begin{thm}\label{transference_Hardy}
If $\lambda $ is a continuous and bounded function on $\R$ such that
$$ \| T_{\lambda} (f) \|_{L^{1,\infty} (\R)} \leq C_1 \| f \|_{h^1 (\R)}, $$
then
$$  \| T_{\lambda |_{\Z}} (g) \|_{L^{1,\infty} (\T)} \leq C_2 \| g \|_{H^1 (\T)}. $$
\end{thm}

Observe that given a Marcinkiewicz multiplier $m$ on the torus, one can construct a Marcinkiewicz multiplier $\lambda$ on $\R$ such that $\lambda|_{\Z} = m$. Indeed, it suffices to take $\lambda$ to be continuous such that $\lambda (n) = m(n)$ for every $n \in \Z$ and affine on the intervals of the form $(n,n+1)$, $n \in \Z$.

In order to use Theorem \ref{transference_Hardy}, let $\psi$ be as above and consider the ``high-frequency'' part $\lambda_+$ of $\lambda$ given by 
$\lambda_+ =  \psi \lambda .$
Note that for every Schwartz function $f$ we may write 
$$T_{\lambda_+} (f) = T_{\lambda} (\widetilde{f}) ,$$
where $\widetilde{f} = T_{\psi}  (f) $.  We thus deduce that
$$ \| T_{\lambda_+} (f) \|_{L^{1,\infty} (\R)} = \| T_{\lambda } (\widetilde{f}) \|_{L^{1,\infty} (\R)}  \lesssim \| \widetilde{f} \|_{H^1 (\R)} \lesssim \| f \|_{h^1 (\R)} ,$$
and hence, Theorem \ref{transference_Hardy} yields that  $ T_{\lambda_+|_{\Z}} $ is bounded from $H^1 (\T)$ to $L^{1,\infty} (\T)$. Since for every trigonometric polynomial $g$ we can write
$$ T_{m } (g)  = T_0 (g)  + T_{\lambda_+|_{\Z}} (g)  ,$$
where $T_0 (g) =m(0)\widehat{g} (0)$, and we have that
$$ \| T_0 (g) \|_{L^{1,\infty} (\T)} = |m(0)||\widehat{g} (0)| \leq \| m\|_{\ell^{\infty} (\Z)} \| g\|_{L^1 (\T)}  \leq \| m\|_{\ell^{\infty} (\Z)} \| g\|_{H^1 (\T)},$$
it follows that
$\| T_{m } (g) \|_{L^{1, \infty} (\T)} \lesssim \| g\|_{H^1 (\T)} .$

\section{A higher-dimensional extension of an inequality due to Zygmund}\label{Application}

In \cite{Zygmund}, Zygmund showed that there exists a constant $C > 0$ such that for every $f\in H^1_A(\T)$, we have
\begin{equation}\label{classical_Zygmund} 
\| S_{\T} (f) \|_{L^1 (\T)} \leq C \| f \|_{L  \log L (\T) } .
\end{equation}
Note that if one removes the assumption that $f \in H^1_A (\T)$, then the Orlicz space $L \log L (\T)$ must be replaced by the smaller space $L \log^{3/2} L (\T)$, see \cite{Bakas}. 

Zygmund's proof again relies on canonical factorisation in $H^p_A(\T)$, but a higher-dimensional extension of (\ref{classical_Zygmund}) can now be obtained from the methods of the previous section.

\begin{proposition}\label{Zygmund_extension}
Given $d \in \N$, there exists a constant $C_d>0$ such that for every analytic trigonometric polynomial $g $ on $\T^d$ one has
\begin{equation}\label{n_Zygmund}
\| S_{\T^d} (g) \|_{L^1 (\T^d)} \leq C_d \| g \|_{L \log^d L (\T^d)}.
\end{equation}
The exponent $r=d$ in the Orlicz space $L \log^d L (\T^d)$ cannot be improved.
\end{proposition}

\begin{proof}
By using Lemma $\ref{KX_P}$ and a Marcinkiewicz-type interpolation argument analogous to the one presented in Section \ref{Proof_of_Thm}, one shows that if $T$ is a sublinear operator that is bounded from $H^1_A (\T)$ to $L^{1,\infty} (\T)$ and bounded from $H^2_A (\T)$ to $L^2 (\T)$, then for every $r \geq 0$ one has
\begin{equation}\label{LlogL_interpolation}
\int_{\T} |T (f) (x)| \log^r (1 + |T(f) (x)|) dx \leq C_r [1 + \int_{\T} |f(x)| \log^{r+1} (1+ |f(x)|)dx ] 
\end{equation}
for every analytic trigonometric polynomial $f$ on $\T$, where $C_r>0$ is a constant depending only on $r$. 

If $T_{\omega_j} = \sum_{k \in \Z} r_k (\omega_j) \Delta_j$ denotes a randomised version of $S_{\T}$, $j=1, \cdots, d$, then $T_{\omega_j}$ maps  $H^1_A (\T)$ to $L^{1,\infty} (\T)$ and  $H^2_A (\T)$ to $L^2 (\T)$ and so, by using (\ref{LlogL_interpolation}) and iteration, one deduces that 
\begin{equation}\label{lin_Zyg}
 \| ( T_{\omega_1} \otimes \cdots \otimes T_{\omega_d} )  (f)  \|_{ L^1 (\T^d)} \leq A_d \| f \|_{L \log^d L (\T^d)}  
 \end{equation}
 for every analytic polynomial $f$ on $\T^d$. Hence, the proof of (\ref{n_Zygmund}) is obtained by using (\ref{lin_Zyg}) and \eqref{multi_Khintchine}.

To prove sharpness for $d=1$, let $N$ be a large positive integer to be chosen later and take $V_{2^N} = 2 K_{2^{N+1}} -K_{2^N}$ to be the de la Vall\'ee Poussin kernel of order $2^N$, where $K_n$ denotes the Fej\'er kernel of order $n$, $K_n (x) = \sum_{|j| \leq n } [1- |j|/(n+1) ] e^{i 2 \pi j x}$. Consider the function $f_N$ by
$$ f_N (x) = e^{i2\pi 2^{N+1} x} V_{2^N} (x). $$
Then, one can easily check that  $f_N \in H^1_A (\T)$, $\Delta_{N+1} (f_N) (x) = \sum_{k=2^N}^{2^{N+1}-1} e^{i 2 \pi k x}$ and $\| f_N \|_{L \log^r L (\T)} \lesssim N^r$. Hence, if we assume that (\ref{classical_Zygmund}) holds for some $L \log^r L (\T)$, then we see that we must have
$$ N \lesssim \| \Delta_{N+1} (f_N) \|_{L^1 (\T)} \leq \| S_{\T} (f_N) \|_{L^1 (\T)} \lesssim \| f_N \|_{L \log^r L (\T)} \lesssim N^r$$
and so, if $N$ is large enough, it follows that $r \geq 1$, as desired. 
To prove sharpness in the $d$-dimensional case, take  $g_N (x_1, \cdots, x_d) = f_N (x_1) \cdots f_N (x_d) $, $f_N$ being as above, and note that
$$ N^d \lesssim \| \Delta_{N+1} (f_N) \|^d_{L^1 (\T)} \leq \|S_{\T^d} (g_N) \|_{L^1 (\T^d)} \lesssim \| g_N \|_{L \log^r L (\T^d)} \lesssim N^r.$$
Hence, by taking $N \rightarrow \infty$, we deduce that $r \geq d$.
\end{proof}

\begin{rmk}
Note that, by using (\ref{LlogL_interpolation}) and  \eqref{multi_Khintchine}, one can actually show that there exists a constant $B_d>0$, depending only on $d$, such that
\begin{equation}\label{weak_type}
 \| S_{\T^d} (f) \|_{L^{1,\infty} (\T^d)} \leq B_d \| f \|_{L \log^{d-1} L (\T^d)} 
 \end{equation}
for every analytic trigonometric polynomial $f$ on $\T^d$. Notice that if we remove the assumption that $f$ is analytic, then the Orlicz space $L \log^{d-1} L (\T^d)$ in (\ref{weak_type}) must be replaced by  $L \log^{3d/2 -1} L (\T^d)$, see \cite{Bakas}.
\end{rmk}

\section{Euclidean variants of Theorem \ref{Main_Theorem}}\label{Euclidean_variants}
In this section we obtain an extension of Pichorides's theorem to the Euclidean setting. Our result will be a consequence of the following variant of Marcinkiewicz-type interpolation on Hardy spaces.

\begin{proposition}\label{PJ_interpolation}
Assume that $T$ is a sublinear operator that satisfies:
\begin{itemize}
\item $ \| T (f) \|_{L^{1,\infty} (\R)} \leq C \| f \|_{L^1 (\R)}  $ for all $f \in H^1_A (\R)$ and
\item $ \| T (f) \|_{L^2 (\R)} \leq C \| f \|_{L^2 (\R)}  $ for all $f \in H^2_A (\R)$,
\end{itemize}
where $C>0$ is an absolute constant.  Then, for every $1<p<2$, $T$ maps $H^p_A (\R)$ to $L^p (\R)$ and moreover,
$$  \| T \|_{H^p_A (\R) \rightarrow L^p (\R)} \lesssim [  (p-1)^{-1} + (2-p)^{-1} ]^{1/p}.$$
\end{proposition}

\begin{proof}
Fix $1<p<2$ and take an $f \in H^p_A (\R)$. From a classical result due to Peter Jones \cite[Theorem 2]{PeterJones} it follows that for every $\lambda>0$ one can write $f = F_{\lambda} + f_{\lambda}$, where $F_{\lambda} \in H^1_A (\R)$, $f_{\lambda} \in H^{\infty}_A (\R)$ and, moreover, there is an absolute constant $C_0 >0$ such that
\begin{itemize}
\item $ \int_{\R} |F_{\lambda} (x) | dx \leq C_0 \int_{ \{ x \in \R:  N(f) (x) > \lambda \} } N(f) (x) dx $ and
\item $\| f_{\lambda} \|_{L^{\infty} (\R)} \leq C_0 \lambda$.
\end{itemize}
Here, $N(f)$ denotes the non-tangential maximal function of $f \in H^p_A (\R)$ given by 
$$ N(f) (x) = \sup_{|x-x'|<t} |(f \ast P_t )(x')|, $$
where, for $t>0$, $P_t(s)=t/(s^2+t^2)$ denotes the Poisson kernel on the real line. Hence, by using the  Peter Jones decomposition of $f$, we have
$$ \| T(f) \|^p_{L^p (\R)} = \int_0^{\infty} p \lambda^{p-1} | \{ x\in \R :  |T(f) (x)| > \lambda/2\}| d \lambda \leq I_1 + I_2 ,$$
where
$$ I_1 = p \int_0^{\infty} \lambda^{p-1} | \{ x\in \R :  |T(F_{\lambda}) (x)| > \lambda/2\}| d \lambda $$
and
$$ I_2 = p \int_0^{\infty} \lambda^{p-1} | \{ x \in \R: |T(f_{\lambda}) (x)| > \lambda/2\}|  d \lambda .$$
We shall treat $I_1$ and $I_2$ separately. To bound $I_1$,  using our assumption on the boundedness of $T$ from $H^1_A(\R)$ to $L^{1,\infty} (\R)$ together with Fubini's theorem, we deduce that there is an absolute constant $C_1 >0$ such that
\begin{equation}\label{est_1}
I_1 \leq  C_1 (p-1)^{-1} \int_{\R} [N(f) (x)]^p dx.
\end{equation}
To bound the second term, we first use the boundedness of $T$ from $H^2_A(\R)$ to $L^2 (\R)$ as follows
$$ I_2 \leq C  \int_0^{\infty}p  \lambda^{p-3} \big( \int_{\R} | f_{\lambda} (x) |^2 dx \big) d \lambda $$
and then we further decompose the right-hand side of the last inequality as $I_{2 ,\alpha} + I_{2 ,\beta}$, where
$$ I_{2, \alpha}  = C   \int_0^{\infty} p  \lambda^{p-3} \big( \int_{ \{ x \in \R:  N(f) (x) > \lambda \}} | f_{\lambda} (x) |^2 dx \big) d \lambda $$
and
$$ I_{2, \beta}  = C  \int_0^{\infty}p  \lambda^{p-3} \big( \int_{ \{ x \in \R:  N(f) (x) \leq \lambda \}} | f_{\lambda} (x) |^2 dx \big) d \lambda. $$
The first term $I_{2, \alpha}$ can easily be dealt with by using the fact that $\| f_{\lambda} \|_{L^{\infty} (\R)} \leq C_0 \lambda$,
$$
I_{2, \alpha} \leq  C'  \int_0^{\infty} p  \lambda^{p-1} | \{ x \in \R : N(f)(x) >  \lambda\} | d \lambda = C' \int_{\R} [N(f) (x)]^p dx ,
$$
where $C' =C_0 C$. 
To obtain appropriate bounds for $I_{2, \beta}$, note that since $ |f_{\lambda}|^2 = |f - F_{\lambda}|^2 \leq 2 |f|^2 + 2 |F_{\lambda}|^2$, one has $I_{2, \beta} \leq I'_{2, \beta} + I''_{2, \beta}$, where
$$ I'_{2, \beta} =  2 C \int_0^{\infty} p \lambda^{p-3} \big( \int_{ \{ x \in \R:   N(f) (x) \leq \lambda \}} |f(x)|^2 d x\big) d \lambda$$
and
$$ I''_{2, \beta} =  2 C \int_0^{\infty} p \lambda^{p-3} \big( \int_{\{ x \in \R:   N(f) (x) \leq \lambda \}} |F_{\lambda}(x)|^2 d x\big) d \lambda  . $$
To handle $I'_{2, \beta} $,  note that since $f\in H^p_A (\R)$ ($1<p<2$) one has $|f (x)| \leq N(f)(x)$ for a.e. $x \in \R$ and hence, by using this fact together with Fubini's theorem, one obtains
$$ I'_{2, \beta} \leq C (2-p)^{-1} \int_{\R} [ N(f) (x) ]^p dx.$$
Finally, for the last term $I''_{2, \beta}$, we note that for a.e. $x$ in $\{ N(f)  \leq \lambda \}$ one has
$$  |F_{\lambda} (x)| \leq |f (x) | + |f_{\lambda} (x)| \leq N(f)(x) + |f_{\lambda} (x)| \leq (1+C_0) \lambda $$
and hence, 
\begin{align*}
I''_{2, \beta}  &\leq C'' \int_0^{\infty} \lambda^{p-2} \big(  \int_{\R} |F_{\lambda} (x)| dx \big) d \lambda \\
&\leq C'' \int_0^{\infty} \lambda^{p-2} \big(  \int_{\{ x \in \R:   N(f) (x) > \lambda \}} N(f) (x) dx \big) d \lambda \leq  C'' (p-1)^{-1} \int_{\R}   [N(f) (x)]^p dx,
\end{align*}
where  $C''= 4(1+C_0) $ and in the last step we used Fubini's theorem.  Since $I_2 \leq I_{2, \alpha}  + I'_{2, \beta}  + I''_{2, \beta} $, we conclude that there is a $C_2>0$ such that
\begin{equation}\label{est_2}
I_2 \leq C_2 [(p-1)^{-1} + (2-p)^{-1}] \int_{\R}   [N(f) (x)]^p dx. 
\end{equation}
It thus follows from (\ref{est_1}) and (\ref{est_2}) that
$$ \| T (f) \|_{L^p (\R)} \lesssim [(p-1)^{-1} + (2-p)^{-1}]^{1/p} \| N(f) \|_{L^p (\R)}.$$
To complete the proof of the proposition note that one has 
\begin{equation}\label{H-L}
\| N(f) \|_{L^p (\R)} \leq C_p \| f \|_{L^p (\R)} \ (f \in H^p_A (\R)),
\end{equation}
where one can take $C_p = A_0^{1/p}$,  $A_0 \geq 1$ being an absolute constant, see e.g. p.278-279 in vol.I in \cite{Zygmund_book}, where the periodic case is presented. The Euclidean version is completely analogous. Hence, if $1<p<2$, one deduces that the constant $C_p$ in (\ref{H-L}) satisfies $C_p \leq A_0$ and so, we get the desired result. \end{proof}

Using the above proposition and iteration, we obtain the following Euclidean version of Theorem \ref{Main_Theorem}.

\begin{thm}\label{Euclidean_Thm}
Let $d \in \N$ be a given dimension. If $T_{m_j}$ is a Marcinkiewicz multiplier operator on $\R$ $(j=1, \cdots, d )$, then 
$$   \|  T_{m_1} \otimes \cdots \otimes T_{m_d}  \|_{H^p_A (\R^d) \rightarrow L^p (\R^d) } \lesssim_{C_{m_1}, \cdots, C_{m_d} } (p-1)^{-d} $$
as $p \rightarrow 1^+$, where $C_{m_j} = \| m_j \|_{L^{\infty} (\R)} + A_{m_j}$, 
$A_{m_j}$ is as in $(\ref{Marcinkiewicz_condition})$, $j=1, \cdots, d$.
\end{thm}

A  variant of Pichorides's theorem on $\R^d$ now follows from Theorem \ref{Euclidean_Thm} and (\ref{multi_Khintchine}). To formulate our result, for $k \in \Z$, define the rough Littlewood-Paley projection $P_k$ to be a multiplier operator given by 
$$\widehat{P_k ( f ) }  = [  \chi_{[2^k, 2^{k+1})}   +  \chi_{(-2^{k+1}, -2^k]}] \widehat{f} .$$
For $d \in  \N$, define the $d$-parameter rough Littlewood-Paley square function $S_{\R^d}$ on $\R^d$ by
$$ S_{\R^d} (f)  = \Big(  \sum_{k_1, \cdots, k_d \in  \Z} |P_{k_1} \otimes \cdots \otimes P_{k_d}  (f) |^2 \Big)^{1/2} $$
for $f$ initially belonging to the class of Schwartz functions on $\R^d$. Arguing as in Subsection \ref{Corollary_of_Thm}, we get a Euclidean version of Corollary \ref{Pichorides_extension} as a consequence of Theorem \ref{Euclidean_Thm}.

\begin{corollary}
For $d\in \N$, one has 
$$ \| S_{\R^d}  \|_{H^p_A (\R^d) \rightarrow L^p (\R^d)  } \sim_d (p-1)^{-d} $$
as $p \rightarrow 1^+$. 
\end{corollary}

\begin{rmk}
The multiplier operators covered in Theorem \ref{Euclidean_Thm} are properly contained in the class of general multi-parameter Marcinkiewicz multiplier operators treated in Theorem $6'$ in Chapter IV of \cite{Stein}.
For a class of smooth multi-parameter Marcinkiewicz multipliers  M. Wojciechowski \cite{Woj} proves that their $L^p (\R^d) \rightarrow L^p(\R^d)$ operator norm is of order $(p-1)^{-d}$ and that they are bounded on the $d$-parameter Hardy space $H^p (\R \times \cdots\times \R)$ for all $1 \leq p \leq 2$. Note that the multi-parameter Littlewood-Paley square function is not covered by this result; see also \cite{Bakas} for more refined negative statements.
\end{rmk}

\end{document}